\documentclass[10pt]{article}
\usepackage{amssymb,latexsym,amsmath,epsfig,amsthm} 

\makeatletter

\renewcommand\section{\@startsection {section}{1}{\z@}
	{-30pt \@plus -1ex \@minus -.2ex}
	{2.3ex \@plus.2ex}
	{\normalfont\normalsize\bfseries\boldmath}}

\renewcommand\subsection{\@startsection{subsection}{2}{\z@}
	{-3.25ex\@plus -1ex \@minus -.2ex}
	{1.5ex \@plus .2ex}
	{\normalfont\normalsize\bfseries\boldmath}}

\renewcommand{\@seccntformat}[1]{\csname the#1\endcsname. }

\makeatother

\newtheorem{theorem}{Theorem}
\newtheorem{lemma}{Lemma}

\theoremstyle{definition}
\newtheorem{definition}{Definition}

\newtheorem{question}{Question}

\begin{document}
	
	\begin{center}
		\textbf{\Large Sumsets with a minimum number of distinct terms }
		\vskip 20pt
		\textbf{Jagannath Bhanja}\\ \smallskip 
		IIITDM Kancheepuram, Chennai\\
		{\tt jagannath@iiitdm.ac.in}\\
	\end{center}
	\vskip 20pt

	\begin{abstract}
		For a set $A$ of $k$ elements from an additive abelian group $G$ and a positive integer $r \leq k$, we consider the set of elements of $G$ that can be written as a sum of $h$ elements of $A$ with at least $r$ distinct elements. We denote this set by $h^{(\geq r)}A$. The set $h^{(\geq r)}A$ generalizes the classical sumsets $hA$ and $h\hat{}A$ for $r=1$ and $r=h$, respectively. As the main result of this article, we give an upper bound for the minimum size of $h^{(\geq r)}A$ over $\mathbb{Z}_m$ for $m \geq 2$. Further, by an observation relating the sumsets $hA$, $h\hat{}A$, and $h^{(\geq r)}A$ we obtain the sharp lower bound on the size of $h^{(\geq r)}A$ and also characterize the set $A$ for which the lower bound on the size of $h^{(\geq r)}A$ is tight over the groups $\mathbb{Z}$ and $\mathbb{Z}_p$, where $p$ is a prime number.
	\end{abstract}
	
	\vspace{10pt}

	\section{Introduction}
	In this article, $p$ stands for a prime number, $\theta$ for the golden ratio $\frac{1+\sqrt{5}}{2}$, $G$ for an additive abelian group, $\mathbb{Z}$ for the group of integers, and $\mathbb{Z}_{m}$ for the group $\{0,1,\dots,m-1\}$, where $m$ is a positive integer. For a given set $A$ of integers, $|A|$ denotes the number of elements of $A$. 
	
	For a non-empty subset $A$ of $G$ and a positive integer $h$, let $hA$ denote the set of elements of $G$ that can be written as a sum of $h$ elements of $A$, and $h\hat{}A$ denote the set of elements of $G$ that can be written as a sum of $h$ distinct elements of $A$. The sets $hA$ and $h\hat{}A$ are called sumsets and restricted sumsets, respectively. One of the important problems in additive combinatorics is to estimate the sumsets $hA$ and $h\hat{}A$ in terms of the size of $A$ and the integer $h$. Such problems are known as direct problems. Some of the early results in this direction are the Cauchy-Davenport theorem \cite{cauchy, dav35, dav47} and the Erd\H{o}s-Heilbronn conjecture (or Dias da Silva-Hamidoune theorem) \cite{SH94}. 
	
	\begin{theorem}[Cauchy-Davenport theorem \cite{cauchy, dav35, dav47}] \label{cauchy-davenport-thm}
		Let $A$ be a non-empty $k$-element set in $\mathbb{Z}_p$. Then, for any positive integer $h$, we have 
		\[|hA| \geq \min\{p,hk-h+1\}.\]
	\end{theorem}
	
	\begin{theorem}[Dias da Silva-Hamidoune theorem \cite{SH94}] \label{silva-hamidoune-thm}
		Let $A$ be a non-empty $k$-element set in $\mathbb{Z}_p$. Then, for any positive integer $h \leq k$, we have 
		\[|h\hat{}A| \geq \min\{p,hk-h^2+1\}.\]
	\end{theorem}
	
	Another important problem in additive number theory is determining the structure of a set $A$ given some information on the size of the sumsets. These are classified as inverse problems. Examples of such results are Vosper's theorem \cite{vosper, vosper-ad}, Nathanson's theorem \cite{nathu95}, and Freiman's $3k-4$ theorem \cite{freiman}. 
	
	\begin{theorem}[Vosper's theorem \cite{vosper, vosper-ad}]\label{vosper-thm}
		Let $A$ be a non-empty $k$-element set in $\mathbb{Z}_p$ with $k\geq 2$. If $h \geq 2$ and $|hA|=hk-h+1\leq p-2$, then $A$ is an arithmetic progression. 
	\end{theorem}
	
	\begin{theorem}[Nathanson \cite{nathu95}]\label{restricted-sumset-inverse-thm}
		Let $A$ be a non-empty $k$-element set of integers with $k \geq 5$. If $2\leq h \leq k-2$ and $|h\hat{}A| = hk-h^2+1$, then $A$ is an arithmetic progression.
	\end{theorem}
	
	\begin{theorem}[Freiman's $3k-4$ theorem \cite{freiman}] \label{Freiman-3k-4-theorem}
		Let $k \geq 3$. Let $A=\lbrace a_{0}, a_{1}, \ldots, a_{k-1}\rbrace$ be a set of integers such that $0=a_{0}<a_{1}< \cdots < a_{k-1}$ and $d(A)=1$. If  $|2A|=2k-1+b \leq 3k-4$, then $A$ is a subset of an arithmetic progression of length at most  $k+b$.
	\end{theorem}
	
	For recent results along this line, see \cite{bhanja-cdm, bhanja-jnt, mistri14, monopoli, yang-chen}. 
	
	In this paper, we consider a sumset that generalizes the sumsets $hA$ and $h\hat{}A$ and also discuss the direct and inverse problems.    
	
	\begin{definition}
		Let $A=\{a_1,a_2,\ldots,a_k\}$ be a non-empty finite set in $G$. Let $r \leq k$ be a positive integer. For integers $h \geq r$, we define 
		\[h^{(\geq r)}A := \left\{ \sum_{i \in I} x_ia_i: I \subset [1, k], |I| \geq r, x_i \geq 1 \text{ for all } i \in I, \text{ and } \sum_{i \in I} x_i=h \right\}.\]
	\end{definition}
	
	Then $h^{(\geq r)}A = hA$ if $r=1$ and $h^{(\geq r)}A = h\hat{} A$ if $r=h$. 
	
	Before we state our main result, we present some background on the problem. 
	
	In 2006, Plagne \cite{plagne-1} (also see \cite{EKP, plagne-2}) extended the Cauchy-Davenport theorem to the group $\mathbb{Z}_m$ by finding the exact value of  $$\mu(\mathbb{Z}_{m},k,h) = \min\{|hA|: A \subset \mathbb{Z}_m, |A|=k\}.$$ 
	
	\begin{theorem}[Plagne \cite{plagne-1}]\label{ekp-thm}
		Let $m, k, h$ be positive integers with $k\leq m$. Then 
		\[\mu(\mathbb{Z}_{m},k,h) = \min\{(h\left\lceil k/d \right\rceil-h+1)d: d\in D(m)\},\]
		where $D(m)$ is the set of positive divisors of $m$. 
	\end{theorem}
	
	Later, in 2016, Bajnok \cite{bajnok} considered the analogous function $\mu\hat{}(\mathbb{Z}_{m},k,h) = \min\{|h\hat{}A|: A \subset \mathbb{Z}_m, |A|=k\}$ and proved an upper bound for $\mu\hat{}(\mathbb{Z}_{m},k,h)$ by estimating the exact size of the set $h\hat{}A_{d}(m,k)$, where $h\hat{}A_{d}(m,k)$ is defined in the following way. For a fixed positive divisor $d$ of $m$, consider the subgroup $H = \{0, m/d, \ldots, (d-1)(m/d)\}$ of $\mathbb{Z}_{m}$ of order $d$. Write $k=ud+v$ for some non-negative integers $u, v$ with $1\leq v\leq d$. Define 
	\begin{equation}\label{specialset-eqn}
		A_{d}(m,k) := \bigcup_{i=0}^{u-1} \bigg\{i+H\bigg\} ~\bigcup~ \bigg\{u+j\cdot \frac{m}{d}: j=0,1,\ldots ,v-1\bigg\}.
	\end{equation} 
	
	The following is the theorem of Bajnok. 
	\begin{theorem}[Bajnok \cite{bajnok}]\label{bajnok-thm}
		Let $h$, $k$, and $m$ be positive integers such that $k\leq m$. Let $d$ be a positive divisor of $m$. Let $A_d(m,k)$ be the set of $k$ elements defined in (\ref{specialset-eqn}).  If $h=k$, then $|h\hat{}A_{d}(m,k)| = 1$, and if $h>k$, then $|h\hat{}A_{d}(m,k)| = 0$. For $1\leq h\leq k-1$, let $v$ and $w$ be the positive remainders of $k$ and $h$ modulo $d$, respectively. Then
		\[|h\hat{}A_{d}(m,k)| =
		\begin{cases}
			\min \left\{m, \left(h\left\lceil \frac{k}{d} \right\rceil-h+1 \right)d,hk-h^2+1 \right\} & \text{if } h\leq \min\{v,d-1\},\\
			\min \{m, hk-h^2+1-\delta_{d}\} & \text{otherwise},
		\end{cases}\]
		where 
		\[\delta_{d} =
		\begin{cases}
			(v-w)w-(d-1) & \text{if } w<v,\\
			(d-w)(w-v)-(d-1) & \text{if } v<w<d,\\
			d-1 & \text{if } v=w=d,\\
			0 & \text{otherwise}.
		\end{cases}\]
	\end{theorem}
	
	The main result of this paper is the following theorem along the lines of Theorem \ref{bajnok-thm}. 
	
	\begin{theorem}\label{main-thm}
		Let $m,k,r$ be positive integers such that $r \leq k \leq m$. For a fixed positive divisor $d$ of $m$, let $A_d(m,k) \subset \mathbb{Z}_m$ be a set of $k$ elements defined in (\ref{specialset-eqn}). Let $v, w$ be the respective positive remainders of $k, r$ modulo $d$. Then, for every integer $h>r$, we have
		\[
		|h^{(\geq r)}A_{d}(m,k)| = 
		\begin{cases} 
			\min \left\{m, \left(h\left\lceil\frac{k}{d}\right\rceil-h+1 \right)d, h(k-1)-r(r-1)+1\right\} & \\ \hspace{5cm} \text{ if } r\leq \min\{v,d-1, k-1\} \\
			\min \{m, h(k-1)-r(r-1)+1-\epsilon_{d}\} & \hspace{-2cm}\text{otherwise}, 
		\end{cases} 
		\]
		where 
		\[\epsilon_{d} =
		\begin{cases}
			(h-r)(v-1)+w(v-w)-(d-1) & \text{if } w<v\\
			(h-r)(v-1)+(d-w)(w-v)-(d-1) & \text{if } v<w<d\\
			(h-r)(v-1)-(d-1) & \text{otherwise}.
		\end{cases}\]
	\end{theorem}
	
	If we set $\mu^{(\geq r)}(\mathbb{Z}_{m},k,h) := \min\{|h^{(\geq r)}A|: A \subset \mathbb{Z}_m, |A|=k\}$, analogous to the functions $\mu(\mathbb{Z}_{m},k,h)$ and $\mu\hat{}(\mathbb{Z}_{m},k,h)$, we obtain the following upper bound from Theorem \ref{main-thm}:
	\begin{equation}
		\mu^{(\geq r)}(\mathbb{Z}_{m},k,h) \leq \min\{|h^{(\geq r)}A_{d}(m,k)|: d\in D(m)\}.
	\end{equation}

	\section{Proof of Main Result}\label{sec4}
	
	In this section we prove Theorem \ref{main-thm}, for which we require the following lemma. 
	
	\begin{lemma}[Bajnok \cite{bajnok}]\label{bajnok-lemma}
		Let $d$ and $t$ be positive integers with $t\leq d-1$, and let $\ell\in \mathbb{Z}_{d}$. Then there is a $t$-element set $L=\{\ell_{1},\ldots,\ell_{t}\}$ in $\mathbb{Z}_{d}$ for which $\ell_{1}+\cdots+\ell_{t}=\ell$.
	\end{lemma}

	\begin{proof}[Proof of Theorem \ref{main-thm}]
		For a fixed positive divisor $d$ of $m$, consider the subgroup $H=\{0, m/d, \ldots, (d-1)(m/d)\}$ of order $d$. Let $k=ud+v$ with $1\leq v\leq d$, and 
		\[A_{d}(m,k)
		=\bigcup_{j=0}^{u-1}\bigg\{j+H\bigg\} ~\bigcup~ \bigg\{u+\ell\cdot \frac{m}{d}: \ell=0,1,\ldots,v-1\bigg\}.\] 
		Since $w$ is the positive remainder of $r$ modulo $d$, we have $r=qd+w$ with some non-negative integer $q$. By definition every element of $h^{(\geq r)}A_{d}(m,k)$ is of the form
		\begin{equation}\label{element-form-eqn}
			(j_{1}+\cdots+j_{h})+(\ell_{1}+\cdots+\ell_{h})\cdot \frac{m}{d},
		\end{equation}
		where $j_{1},\ldots,j_{h} \in \{0,1,\ldots,u\}$ and $\ell_{1},\ldots,\ell_{h} \in \{0,1,\ldots,d-1\}$, with the following two added conditions:
		\begin{enumerate}
			\item[(i)] when any of the $j$-indices equals $u$, the corresponding $\ell$-index is at most $v-1$,
			\item[(ii)] there are at least $r$ distinct $j+\ell\cdot (m/d)$ in every sum of the form (\ref{element-form-eqn}). 
		\end{enumerate}
		
		Let us denote the least value of $j_{1}+\cdots+j_{h}$ in (\ref{element-form-eqn}) by $j_{min}$ and the largest value of $j_{1}+\cdots+j_{h}$ in (\ref{element-form-eqn}) by $j_{max}$. Then 
		\[j_{min}
		=(h-r+d) \cdot 0+d \cdot 1+d \cdot 2+\cdots+d \cdot (q-1)+w\cdot q
		=q(r+w-d)/2.\]
		On the other hand, to calculate $j_{max}$ we consider the following four possible cases. If $q=0$ and $w>v$, then $r=w>v$ and $1\leq w-v\leq d-1$. Therefore
		\[j_{max} = (h-r+v)\cdot u+(w-v)\cdot (u-1)
		= hu-r+v.\]
		If $q=0$ and $w\leq v$, then $r=w\leq v$. Therefore
		\[j_{max} = h\cdot u = hu-r+w.\]
		If $q\geq 1$ and $w>v$, then by writing $r=v+qd+(w-v)$, where $1\leq w-v\leq d-1$, we get 
		\begin{align*}
			j_{max}
			&= (h-r+v)\cdot u+d\cdot(u-1)+\cdots
			+d\cdot(u-q)+(w-v)\cdot(u-q-1).
		\end{align*}
		After simplification we get 
		\[j_{max} = hu-r+vq-\frac{q(r+w-d)}{2}+v.\]
		Finally, if $q\geq 1$ and $w\leq v$, then by writing $r=v+(q-1)d+(d+w-v)$, where $1\leq d+w-v\leq d$, we get
		\begin{align*}
			&j_{max} \\ 
			&= (h-r+v)\cdot u+d\cdot(u-1)+\cdots
			+d\cdot(u-q+1)+(d+w-v)\cdot(u-q) \\
			&= (h-r+v)\cdot u+d\cdot(u-1)+\cdots+d\cdot(u-q)
			+(w-v)\cdot(u-q-1) + (w-v) \\
			&= hu-r+vq-\frac{q(r+w-d)}{2}+w.
		\end{align*}
		We can combine the above four possible values of $j_{max}$ and write in the following unified form
		\[j_{max} = hu-r+vq-\frac{q(r+w-d)}{2}+\min\{v,w\}.\]
		
		Since $j=j_{1}+\cdots+j_{h}$ can assume any integer value between $j_{min}$ and $j_{max}$, the sumset $h^{(\geq r)}A_{d}(m,k)$ lies in exactly  $\min\{m/d, j_{max}-j_{min}+1\}$ cosets of $H$. Thus,
		\begin{equation}\label{upperbound-equation}
			|h^{(\geq r)}A_{d}(m,k)| \leq \min \{m,(j_{max}-j_{min}+1)d\},
		\end{equation}
		where after simplification we can write 
		\[(j_{max}-j_{min}+1)d
		= h(k-1)-r(r-1)-(h-r)(v-1)-w(v-w)+d\cdot\min\{0,v-w\}+d.\]
		
		Equation (\ref{upperbound-equation}) gives an upper bound for the size of $h^{(\geq r)}A_{d}(m,k)$. However, to compute the exact value of $|h^{(\geq r)}A_{d}(m,k)|$ we need to consider the following three cases.
		
		\noindent\textbf{Case 1.} $r=k$.  
		
		In this case $u=q$ and $v=w$. Therefore,
		\begin{align*} 
			&|h^{(\geq r)}A_{d}(m,k)| \\
			&= \min\{m,(j_{max}-j_{min}+1)d\}\\
			&= \min\{m, h(k-1)-r(r-1)-(h-r)(v-1)-w(v-w)+d \cdot \min\{0,v-w\} \\ 
			&\quad \hspace{11cm} +d\}\\
			&= \min\{m, h(k-1)-r(r-1)-(h-r)(v-1)+d\}.
		\end{align*}
		
		\noindent\textbf{Case 2.} $r=1$. 
		
		In this case $h^{(\geq r)}A_d(m,k)$ is simply $hA_d(m,k)$. So,
		\[h^{(\geq r)}A_{d}(m,k)
		= \bigcup_{j=0}^{hu-1} \bigg\{j+H\bigg\} 
		~\bigcup~ \bigg\{hu+\ell\cdot\frac{m}{d}: \ell=0,1,\ldots,h(v-1)\bigg\}.\]
		Hence,
		\begin{align*}
			|h^{(\geq r)}A_{d}(m,k)|
			&= \min\{m,hud+\min\{d,hv-h+1\}\}\\
			&= \min\{m,(hu+1)d,h(ud+v)-h+1\}\\
			&= \min\{m,(h\lceil k/d \rceil-h+1)d, h(k-1)+1\}\\
			&= \min\{m,(h\lceil k/d \rceil-h+1)d, h(k-1)-r(r-1)+1\}.
		\end{align*}
		
		\noindent\textbf{Case 3.} $2 \leq r \leq k-1$. 
		
		\noindent{\bf Subcase 1.} Assume that $r\leq v$ and $r<d$. Then $q=0$, $r=w<d$, and $r=w\leq v$. Therefore, $j_{min}=0$ and $j_{max}=hu$. Thus, by Lemma \ref{bajnok-lemma}, we get
		\begin{align*}
			h^{(\geq r)}A_{d}(m,k) 
			&= \bigcup_{j=0}^{hu-1} \bigg\{j+H\bigg\} ~\bigcup~ \bigg\{hu+\ell\cdot \frac{m}{d}: \ell=\frac{r(r-1)}{2},\frac{r(r-1)}{2}+1,\ldots, \\ 
			&\hspace{7cm} h(v-1)-\frac{r(r-1)}{2}
			\bigg\}.
		\end{align*}
		Hence,
		\begin{align*}
			|h^{(\geq r)}A_{d}(m,k)|
			&= \min\{m,hud+\min\{d,h(v-1)-r(r-1)+1\}\}\\
			&= \min\{m,(hu+1)d,hud+h(v-1)-r(r-1)+1\}\\
			&= \min\{m,(h\lceil k/d\rceil-h+1)d,h(k-1)-r(r-1)+1\}.
		\end{align*}
		
		\noindent{\bf Subcase 2.} Assume that $r=v=d$. Then $q=0$ and $w=r=d$. This implies $j_{min}=0$ and $j_{max}=hu$. Further, since $r=d$, $k=ud+v=(u+1)d$, and $r<k$, we have $u\geq 1$. Therefore,
		\begin{align*}
			&h^{(\geq r)}A_{d}(m,k)\\
			&=\bigg\{\ell\cdot \frac{m}{d}: \ell=\frac{d(d-1)}{2},\frac{d(d-1)}{2}+1,\ldots,h(d-1)-\frac{d(d-1)}{2}\bigg\} \bigcup_{j=1}^{hu-1}\bigg\{j+H\bigg\} \\  
			&\hspace{1.5cm} ~\bigcup~ \bigg\{hu+\ell\cdot \frac{m}{d}: \ell=\frac{d(d-1)}{2},\frac{d(d-1)}{2}+1,\ldots,h(d-1)-\frac{d(d-1)}{2}\bigg\}.
		\end{align*}
		Hence,
		\begin{align*}
			|h^{(\geq r)}A_{d}(m,k)|
			&= \min\{m, (hu-1)d+2\cdot \min\{d,h(d-1)-d(d-1)+1\}\}\\
			&= \min\{m, (hu+1)d,(hu-1)d+2(h-d)(d-1)+2\}. 
		\end{align*}
		As $h\geq r+1=d+1$ and $d\geq 1$, we have $(hu-1)d+2(h-d)(d-1)+2 \geq (hu+1)d$. Therefore, 
		\begin{align*}
			|h^{(\geq r)}A_{d}(m,k)|
			= \min\{m,(hu+1)d\}
			= \min\{m,(h\lceil k/d\rceil-h+1)d\}.
		\end{align*}
		But, $(h\lceil k/d\rceil-h+1)d=h(k-1)-r(r-1)-(h-r)(v-1)+d$, thus, we get
		\begin{align*}
			|h^{(\geq r)}A_{d}(m,k)|
			= \min\{m,h(k-1)-r(r-1)-(h-r)(v-1)+d\}.
		\end{align*}
		
		\noindent{\bf Subcase 3.} Assume that $r>v$, $w\neq d$, and $w\neq v$. Then, by Lemma \ref{bajnok-lemma}, we have
		\[h^{(\geq r)}A_{d}(m,k) = \bigcup_{j=j_{min}}^{j_{max}} \bigg\{j+H\bigg\}.\]
		Hence,
		\begin{align*}
			&|h^{(\geq r)}A_{d}(m,k)|\\
			&= \min\{m, (j_{max}-j_{min}+1)d\}\\
			&= \min\{m, h(k-1)-r(r-1)-(h-r)(v-1)-w(v-w)+d\cdot\min\{0,v-w\}\\ 
			&\quad \hspace{11cm} +d\}.
		\end{align*}	
		
		\noindent{\bf Subcase 4.} Assume that $r>v$, $w=d$, and $w\neq v$. Then $r=(q+1)d$ and $k=ud+v$ with $v<d$. This implies $j_{min}=\frac{q(q+1)d}{2}$ and $j_{max}=hu$. Thus,  
		\begin{align*}
			&h^{(\geq r)}A_{d}(m,k) \\
			&= \bigg\{\frac{q(q+1)d}{2}+\ell\cdot \frac{m}{d}: \ell=\frac{(q+1)d(d-1)}{2},\frac{(q+1)d(d-1)}{2}+1,\ldots,\\
			&\hspace{4cm} \frac{(q+1)d(d-1)}{2}+(h-r)(d-1)\bigg\} \bigcup_{j=\frac{q(q+1)d}{2}+1}^{hu} \bigg\{j+H\bigg\}.
		\end{align*}
		Hence,
		\begin{align*}
			&|h^{(\geq r)}A_{d}(m,k)| \\
			&= \min\{m, \min\{d, (h-r)(d-1)+1\}+(j_{\max}-j_{\min})d\}\\
			&= \min\{m, (j_{\max}-j_{\min}+1)d, (j_{\max}-j_{\min})d+(h-r)(d-1)+1\}\\
			&= \min\{m, h(k-1)-r(r-1)-(h-r)(v-1)+d,h(k-1)-r(r-1) \\
			&\hspace{8cm} -(h-r)(v-d)+1\}.
		\end{align*}
		Note that \[ h(k-1)-r(r-1)-(h-r)(v-1)+d \leq h(k-1)-r(r-1)-(h-r)(v-d)+1 \] if and only if \[ (h-r)(d-1) \geq d-1. \] 
		As $h>r$, the latter inequality is true. Hence, 
		\begin{align*}
			|h^{(\geq r)}A_{d}(m,k)| = \min\{m, h(k-1)-r(r-1)-(h-r)(v-1)+d\}.
		\end{align*}    
		
		\noindent{\bf Subcase 5.} Assume that $r>v$, $w\neq d$, and $w=v$. Then $r=qd+v$ and $k=ud+v$ with $v=w<d$. This implies $j_{min}=q(r+w-d)/2$ and $j_{max}=hu-r+v-\frac{q(r-v-d)}{2}$. Thus,  
		\begin{align*}
			&h^{(\geq r)}A_{d}(m,k) \\
			&= \bigcup_{j=j_{\min}}^{j_{\max}-1} \bigg\{j+H\bigg\} ~\bigcup~ \bigg\{j_{max}+\ell\cdot \frac{m}{d}: \ell=\frac{d(d-1)q}{2}+\frac{v(v-1)}{2},\\ 
			&\quad \hspace{1.5cm} \frac{d(d-1)q}{2}+\frac{v(v-1)}{2}+1,\ldots,\frac{d(d-1)q}{2}+\frac{v(v-1)}{2}+(h-r)(d-1)\bigg\}.	
		\end{align*}
		Hence,
		\begin{align*}
			&|h^{(\geq r)}A_{d}(m,k)| \\
			&= \min\{m, (j_{\max}-j_{\min})d+\min\{d, (h-r)(d-1)+1\}\}\\
			&= \min\{m, (j_{\max}-j_{\min}+1)d, (j_{\max}-j_{\min})d+(h-r)(d-1)+1\}\\
			&= \min\{m, (j_{\max}-j_{\min}+1)d\}\\
			&= \min\{m, h(k-1)-r(r-1)-(h-r)(v-1)+d\}.
		\end{align*}
		
		\noindent{\bf Subcase 6.} Assume that $r>v$ and $w=v=d$. Then $r=(q+1)d$ and $k=(u+1)d$. Thus,  
		\begin{align*}
			&h^{(\geq r)}A_{d}(m,k) \\
			&= \bigg\{j_{\min}+\ell\cdot \frac{m}{d}: \ell=\frac{(q+1)d(d-1)}{2},\ldots,\frac{(q+1)d(d-1)}{2}+(h-r)(d-1)\bigg\} \\
			&\qquad \bigcup_{j=j_{\min}+1}^{j_{\max}-1}\bigg\{j+H\bigg\}~ \bigcup~ \bigg\{j_{\max}+\ell\cdot \frac{m}{d}: \ell=\frac{(q+1)d(d-1)}{2},\ldots, \\ 
			&\hspace{7cm} \frac{(q+1)d(d-1)}{2}+(h-r)(d-1)\bigg\}.	
		\end{align*}
		Hence,
		\begin{align*}
			&|h^{(\geq r)}A_{d}(m,k)| \\
			&= \min\{m, (j_{\max}-j_{\min}-1)d + 2\cdot \min\{d, (h-r)(d-1)+1\}\}\\
			&= \min\{m, (j_{\max}-j_{\min}+1)d, (j_{\max}-j_{\min}-1)d+2(h-r)(d-1)+2\}\\ 
			&= \min\{m, h(k-1)-r(r-1)-(h-r)(v-1)+d, h(k-1)-r(r-1) \\ 
			&\hspace{7.5cm} +(h-r-1)(d-1)+1\}\\
			&= \min\{m, h(k-1)-r(r-1)-(h-r)(v-1)+d\}.
		\end{align*}
	\end{proof}

	\section{An Observation and its Consequences}\label{sec2} 
	
	Recall that in the introduction we defined the sumsets $hA$ and $h\hat{}A$ for positive integers $h$. Here we add the convention that $0A=0\hat{}A=\{0\}$. As every element of $h^{(\geq r)}A$ is a sum of $h$ terms of $A$ with at least $r$ of them distinct, we have the following lemma. 
	
	\begin{lemma}\label{relation-thm}
		Let $A$ be a non-empty $k$-element set in $G$. Let $r$ be a positive integer such that $r \leq k$. Then, for every positive integer $h \geq r$, we have
		\begin{equation}\label{relation-eqn}
			h^{(\geq r)}A = (h-r)A+r\hat{}A.
		\end{equation}
	\end{lemma}
	
	We note that this lemma does not provide a simple proof of Theorem \ref{main-thm}. However, we can use it to obtain some direct and inverse results for $h^{(\geq r)}A$ (see Theorems \ref{direct-theorem-modp}, \ref{inverse-theorem-modp}, and \ref{theorem-in-integers}). We can also use Lemma \ref{relation-thm} to obtain a Freiman's $3k-4$-theorem type result for $h^{(\geq r)}A$. We start with the following direct theorem. 
	
	\begin{theorem}\label{direct-theorem-modp}
		Let $k,r$ be positive integers with $r \leq k$. Then, for every positive integer $h \geq r$ and non-empty $k$-element set $A$ in $\mathbb{Z}_p$, we have
		\begin{equation}\label{eqn-direct-theorem-modp}
			|h^{(\geq r)}A| \geq \min\{p, h(k-1)-r(r-1)+1\}.
		\end{equation}
		This lower bound is sharp.
	\end{theorem}
	
	To prove this theorem we shall use Theorem \ref{cauchy-davenport-thm} together with the following theorem, which is another version of the Cauchy-Davenport theorem.
	
	\begin{theorem}[\cite{nathu}] \label{cdthm}
		Let $A, B$ be non-empty subsets of $\mathbb{Z}_p$, where $p$ is a prime number. Then
		\[|A+B| \geq \min\{p, |A|+|B|-1\}.\]
	\end{theorem}
	
	\begin{proof}[Proof of Theorem \ref{direct-theorem-modp}]
		The case $k=1$ is obvious. Assume that $k \geq 2$. By Lemma \ref{relation-thm}, we have $h^{(\geq r)}A = (h-r)A+r\hat{}A$. Then, by Theorem \ref{cdthm},  
		\begin{align*}
			|h^{(\geq r)}A| 
			= |(h-r)A+r\hat{}A|
			\geq \min\{p, |(h-r)A|+|r\hat{}A|-1\}.
		\end{align*}
		Now, by applying Theorem \ref{cauchy-davenport-thm} on $(h-r)A$ and Theorem \ref{silva-hamidoune-thm} on $r\hat{}A$, we obtain 
		\begin{align*}
			|h^{(\geq r)}A| 
			&\geq \min\{p, |(h-r)A|+|r\hat{}A|-1\}\\
			&\geq \min\{p, \min\{p, (h-r)(k-1)+1\}+\min\{p, r(k-r)+1\}-1\}\\
			&\geq \min\{p, (h-r)(k-1)+r(k-r)+1\}\\
			&= \min\{p, h(k-1)-r(r-1)+1\}.
		\end{align*}
		This establishes the lower bound in Equation (\ref{eqn-direct-theorem-modp}).
		
		Choose a prime number $p$ and positive integers $h, r, k$ such that $r \leq \min\{h, k\}$ and $h(k-1)-r(r-1)+1\leq p$. Let $A=[1, k] \subset \mathbb{Z}_p$. Then
		\[h^{(\geq r)}A \subset [(h-r+1) \cdot 1+2+3+\cdots+r, (h-r+1) \cdot k+(k-1)+(k-2)+\cdots+(k-r+1) ].\]
		Therefore,
		\[|h^{(\geq r)}A| \leq h(k-1)-r(r-1)+1.\]
		This upper bound together with Equation (\ref{eqn-direct-theorem-modp}) shows that $|h^{(\geq r)}A| = h(k-1)-r(r-1)+1$. Hence, the lower bound in Equation (\ref{eqn-direct-theorem-modp}) is sharp.
	\end{proof}
	
	The following theorem gives the structure of set $A$ for which the bound in Theorem \ref{direct-theorem-modp} is tight.
	
	\begin{theorem}\label{inverse-theorem-modp}
		Let $k, r$ be positive integers with $k \geq 2$ and $r \leq k$. Let $h \geq r+2$ be a positive integer and $A$ be a non-empty $k$-element set in $\mathbb{Z}_p$ with
		\[ |h^{(\geq r)}A| = h(k-1)-r(r-1)+1 \leq p-2. \]
		Then $A$ is an arithmetic progression.
	\end{theorem}
	
	\begin{proof} As any set of two elements is an arithmetic progression, we shall assume that $k \geq 3$. Observe that
		\begin{align*}
			\max\{(h-r)(k-1)+1, r(k-r)+1\} 
			&\leq (h-r)(k-1)+ r(k-r)+1 \\ 
			&=h(k-1)-r(r-1)+1 \\
			&\leq p-2.
		\end{align*}
		As it is given that 
		\[ |h^{(\geq r)}A| = h(k-1)-r(r-1)+1, \] from the proof of Theorem \ref{direct-theorem-modp}, we obtain 
		\[ |(h-r)A|=\min\{p, (h-r)(k-1)+1\}=(h-r)(k-1)+1 \leq p-2. \] 
		Then, Theorem \ref{vosper-thm} implies that the set $A$ is an arithmetic progression.  
	\end{proof}
	
	One can prove the following theorem using similar arguments as in Theorem \ref{direct-theorem-modp} and Theorem \ref{inverse-theorem-modp}.  
	
	\begin{theorem}\label{theorem-in-integers}
		Let $k,r$ be positive integers with $r \leq k$. Then for every positive integer $h \geq \max\{r, 2\}$ and non-empty $k$-element set $A$ of integers, we have
		\begin{equation}\label{eqn-direct-theorem-in-integers}
			|h^{(\geq r)}A| \geq h(k-1)-r(r-1)+1.
		\end{equation}
		
		Furthermore, if $|h^{(\geq r)}A| = h(k-1)-r(r-1)+1$ and 
		\begin{enumerate}
			\item $k \geq 2$ in case $h \geq r+2$,
			\item $k \geq 5$ and $2 \leq r \leq k-2$ in case $h=r$ or $h=r+1$, 
		\end{enumerate}
		then $A$ is an arithmetic progression. 
	\end{theorem}
	
	From Theorem \ref{theorem-in-integers} it follows that the set $A$ is an arithmetic progression if the lower bound for $|h^{(\geq r)}A|$ is tight. A natural question that arises here is: up to what deviation from the lower bound in Theorem \ref{theorem-in-integers} can one cover the set $A$ inside a ``small'' arithmetic progression? To answer such a question, one possible way is to express the size of $h^{(\geq r)}A$ in terms of the size of $(h-1)^{(\geq r)}A$ or lower order sumsets. Theorem \ref{extended-theorem-1} is one such result in this direction, which directly follows from Lemma \ref{relation-thm} and the following theorem of Lev \cite{lev}.
	
	\begin{theorem}[Lev \cite{lev}] \label{Lev-theorem-1996}
		Let $k \geq 3$. Let $A=\lbrace a_{0}, a_{1}, \ldots, a_{k-1}\rbrace$ be a set of integers such that $0=a_{0}<a_{1}< \cdots < a_{k-1}$ and $\gcd(A)=1$. Then, for $h \geq 2$, we have
		\[|hA| \geq |(h-1)A| + \min \{ a_{k-1}, h(k-2)+1\}. \]
	\end{theorem}
	
	\begin{theorem}\label{extended-theorem-1}
		Let $r,k$ be positive integers such that $k \geq 3$ and $r \leq k$. Let $A=\lbrace a_{0}, a_{1}, \ldots, a_{k-1}\rbrace$ be a set of $k$ integers such that $0=a_{0}<a_{1}< \cdots < a_{k-1}$ and $\gcd(A)=1$. Then, for $h \geq r+2$, we have
		\[ |h^{(\geq r)}A| \geq |r\hat{}A| + |(h-r-1)A| + \min \{ a_{k-1}, (h-r)(k-2)+1\}. \]
	\end{theorem}

	\section{Further Remarks}
	
	Another sumset that may come immediately to the reader's mind is the sumset $h^{(\leq r)}A$, which contains those elements of the group $G$ that are the sum of $h$ elements of $A$ with at most $r$ distinct elements. The sumset $hA$ is a special case of $h^{(\leq r)}A$ for $r=h$. 
	
	Let $A=\{a_1, a_2, \ldots, a_k\}$ be a set of $k$ integers. Then for any positive integer $h$ the sumset $hA$ contains at least the following $hk-h+1$ elements written in an increasing order:
	\begin{align*}
		& ha_1 < (h-1)a_1+a_2 < \cdots <a_1+(h-1)a_2 <\\
		& ha_2 < (h-1)a_2+a_3 < \cdots <a_2+(h-1)a_3 <\\
		& \vdots\\
		& ha_{k-1} < (h-1)a_{k-1}+a_k < \cdots < a_{k-1}+(h-1)a_k <\\
		& ha_k.
	\end{align*}
	Observe that to get the minimum number of elements (which is $hk-h+1$) in the sumset $hA=h^{(\leq h)}A$, we only required elements in $hA$ which are sums of at most two distinct elements of $A$ (see the list of elements above). Thus, the sumset $h^{(\leq r)}A$ contains at least $hk-h+1$ distinct elements for all $r \geq 2$. Therefore, we ask the following question for integers $r \geq 2$.
	
	\begin{question}
		Is it always the case that the minimum size of both the sumsets $hA$ and $h^{(\leq r)}A$ are the same in any abelian group $G$, where the minimum runs over the subsets of $G$ of size (say) $k$?
	\end{question}

	\vskip20pt\noindent {\bf Acknowledgement.} 
	Thanks to Professor Georges Grekos for reading the article carefully, pointing out some typos, and making suggestions that improved the presentation of the article.

	\end{document}